\documentclass[11pt]{amsart}
\usepackage{bm}
\usepackage{fullpage}
\usepackage{amssymb}
\usepackage{amsmath,amsfonts,amsthm}
\usepackage{hyperref}
\usepackage{color}
\usepackage{cite}
\usepackage{tikz}
\definecolor{uuuuuu}{rgb}{0.26666666666666666,0.26666666666666666,0.26666666666666666}

\newtheorem{theorem}{Theorem}[section]

\newtheorem{lemma}[theorem]{Lemma}

\newtheorem{definition}[theorem]{Definition}
\newtheorem{claim}[theorem]{Claim}

\title{The $m$-bipartite Ramsey number $BR_m(H_1,H_2)$}
\author[1]{Yaser Rowshan }

\email{y.rowshan@iasbs.ac.ir,~~~y.rowshan.math@gmail.com}

\begin{document}
	\maketitle
	
	\begin{abstract}In a $(G^1,G^2)$ coloring of a graph $G$, every edge of $G$ is in  $G^1$ or $G^2$. For two bipartite graphs $H_1$ and $H_2$, the bipartite Ramsey number $BR(H_1, H_2)$ is the least  integer $b\geq 1$, such that for every $(G^1, G^2)$ coloring of the complete bipartite graph $K_{b,b}$,  results in either  $H_1\subseteq G^1$ or $H_2\subseteq G^2$. As another  view, for bipartite graphs $H_1$ and $H_2$ and a positive integer $m$, the $m$-bipartite Ramsey number $BR_m(H_1, H_2)$ of $H_1$ and $H_2$ is the least integer $n$, such that every subgraph $G$ of $K_{m,n}$ results in $H_1\subseteq G$ or $H_2\subseteq \overline{G}$. The size of $m$-bipartite Ramsey number  $BR_m(K_{2,2}, K_{2,2})$, the size of $m$-bipartite Ramsey number $BR_m(K_{2,2}, K_{3,3})$ and the size of $m$-bipartite Ramsey number $BR_m(K_{3,3}, K_{3,3})$ have been computed in several articles up to now. In this paper  we determine the exact value of  $BR_m(K_{2,2}, K_{4,4})$ for each $m\geq 2$.
	\end{abstract}
	
	\section{Introduction}
	In a $(G^1,G^2)$ coloring of a graph $G$, every edge of $G$ is in  $G^1$ or $G^2$. For two graphs $G$ and $H$, the Ramsey number $R(G, H)$ is the smallest
	positive integer $n$, such that for every $(G^1, G^2)$ coloring of the complete graph $K_n$,  results in either  $G\subseteq G^1$ or $H\subseteq G^2$. Frank Harary always liked this notation! All	such Ramsey numbers $R(G, H)$ exist as well, for if $G$ has order $n$ and $H$ has order $m$, then $R(G, H) \leq  R(K_n, K_m)$.
	In \cite{beinere1976bipartite}, Beineke and Schwenk introduced a bipartite version of Ramsey numbers. For given bipartite graphs $H_1$ and $H_2$, the  bipartite Ramsey number $BR(H_1,H_2)$ is defined as the smallest positive integer $b$, such that any subgraph $G$ of $K_{b,b}$ results in either a $H_1\subseteq G$ or $H_2\subseteq \overline{G}$. One can refer to \cite{faudree1975path, gyarfas1973ramsey, hattingh1998star}, \cite{ raeisi2015star, bucic2019multicolour, bucic20193, lakshmi2020three}, \cite{ hatala2021new, rowshan2021size, rowshan2021proof, gholami2021bipartite} , \cite{goddard2000bipartite, longani2003some, rowshan2021multipartite, rowshan1} and their references for further studies.
	
	Suppose that  $(G^1,G^2)$ be any  $2$-edges coloring  of $K_{m,n}$ when $m\neq n$, also let $H_1$ and $H_2$ are two bipartite graphs, for each $m\geq 1$, the $m$-bipartite Ramsey number $BR_m(H_1, H_2)$ of $H_1$ and $H_2$ is the least integer $n$, so that any  $(G^1,G^2)$ coloring of $K_{m,n}$ results in either a $H_1\subseteq G^1$ or $H_2\subseteq G^2$. Z. Bi, G. Chartrand and P. Zhang in \cite{bi2018another} evaluate $BR_m(K_{2,2},  K_{3,3})$ for each $m\geq 2$. Recently Y. Rowshan and M. Gholami, by another simple proof,  evaluated $BR_m(K_{2,2},  K_{3,3})$ in \cite{2022arXiv220112844R}. Z. Bi, G. Chartrand, and P. Zhang in \cite{chartrand2021new} evaluate $BR_m(K_{3,3},  K_{3,3})$.
	
	In this article  we determine the exact value of $BR_m(K_{2,2}, K_{4,4})$ for each $m\geq 2$. In particular we prove the following theorem:
	\begin{theorem}\label{M.th}
		Let $m\geq 2 $ be a positive integer. Then:
		\[
		BR_m(K_{2,2},K_{4,4})= \left\lbrace
		\begin{array}{ll}	
			\text{does not exist}, & ~~~if~~~m=2,3,4,\vspace{.2 cm}\\
			26 & ~~~if~~~m=5,\vspace{.2 cm}\\
			22 & ~~~if~~~ m=6,7,\vspace{.2 cm}\\
			16 & ~~~if~~~ m=8,\vspace{.2 cm}\\
			14 & ~~~if~~~ m\in \{9,10\ldots,13\}.\vspace{.2 cm}\\
		\end{array}
		\right.
		\]	
	\end{theorem} 
	There is a familiar problem corresponding to the Ramsey number $R(K_3, K_3)$, which is stated as follows: What is the last number of people who must be present at a meeting, where every two people are either acquaintances or strangers, so that there are three among them who are either mutual acquaintances or mutual strangers? Since $R(K_3, K_3) = 6$, the answer to this question is $6$. On the other hand, for a gathering of people, five of whom are men, what is the smallest number of women who must also be present at the meeting so that there are four among them, two men and two women, where each man is an acquaintance of each woman, or there are eight among them, four men and four women, where each man is a stranger of each woman. By Theorem \ref{M.th}, as $BR_5(K_{2,2}, K_{4,4}) = 26$,
	the required number of women to be present is $26$.
	\section{Preparations}  
	Suppose that $G=(V(G),E(G))$ is a graph. The degree of a vertex $v\in V(G)$ is denoted by $\deg_G(v)$. For each $v\in V(G)$, $N_G(v) =\{ u\in V(G), ~~vu\in E(G)\}$. The maximum and  minimum degrees of  $V(G)$ are defined by $\Delta(G)$ and $\delta(G)$, respectively. 
Let $G[ X, Y]$ (or in short $[X, Y ]$), be a bipartite graph with bipartition $X \cup Y $. Suppose that $E(G[ X', Y'])$, for short  $E[ X', Y']$, denotes the edge set of $G[ X', Y']$. We use $\Delta (G_X)$ to denote the maximum degree of  vertices in  part $X$ of $G$.

 Let $(G^1,G^2)$ be a $2$-coloring of a graph $G$, then every edge of $G$ is in  $G^1$ or  $G^2$. For given  graphs $G$, $H_1$, and $H_2$, we say $G$ is $2$-colorable to $(H_1, H_2)$ if there exists a $2$-edge decomposition of $G$, say $(G^1, G^2)$ where $H_i\nsubseteq G^i$ for each $i=1,2$. We use $G\rightarrow (H_1, H_2)$, to show that  $G$  is $2$-colorable to $(H_1, H_2)$.
	\begin{definition}
		The Zarankiewicz number $z((m_1, m_2), K_{n_1, n_2})$ is defined to be the maximum number of edges in any subgraph $G$ of the complete bipartite graph $K_{m_1,m_2}$, such that $G$ does not contain  $K_{n_1,n_2}$.
	\end{definition}

	\begin{lemma}\label{l0}{(\cite{collins2016zarankiewicz})}The following results on  $z((m,n), K_{t,t})$ are true: 	
		\begin{itemize}	
			\item[$\bullet$] $z((7,14), K_{2,2})\leq  31$.
			\item[$\bullet$] $z((7,16), K_{2,2})\leq  34$.
			\item[$\bullet$] $z((8,14), K_{2,2})\leq  35$.
			\item[$\bullet$] $z((8,16), K_{2,2})\leq  38$.	
			\item[$\bullet$] $z((8,16), K_{4,4})\leq  90$.
			\item[$\bullet$] $z((9,14), K_{2,2})\leq  39$.
			\item[$\bullet$] $z((9,14), K_{4,4})\leq  88$.
			\item[$\bullet$] $z((10,14), K_{2,2})\leq 42$.
			\item[$\bullet$] $z((10,14), K_{4,4})\leq 97$.

		\end{itemize}
		\begin{proof}
			By using the bounds in Table $5$ of \cite{collins2016zarankiewicz} and Table C.0 of \cite{collins2015bipartite}, the  proposition holds.
		\end{proof}
	\end{lemma}
	Hattingh and Henning in \cite{hattingh1998bipartite} prove that the next theorem:
	\begin{theorem}\label{th1}\cite{hattingh1998bipartite} $BR(K_{2,2}, K_{4,4})=14$.
	\end{theorem}

	\begin{lemma}\label{l1} Assume that $G$ be a subgraph of $K_{|X|,|Y|}$, where $|X|=m\geq 5$, and $|Y|=n\geq 8$. If $\Delta (G_X)\geq 8$, then either $K_{2,2} \subseteq G$ or $K_{4,4} \subseteq \overline{G}$.
	\end{lemma}
	\begin{proof}
		Without loss of generality	(W.l.g), let $\Delta (G_X)=8$ and  $N_G(x)=Y'$, where  $|Y'|\geq= 8$. Suppose that $K_{2,2} \nsubseteq G$, hence 	 $|N_G(x')\cap Y'|\leq 1$ for any $x'\in X\setminus\{x\}$. So, since $|X|\geq 5$ and $|Y'|= 8$, then it is easy to check that $K_{4,4} \subseteq \overline{G}[X\setminus\{x\}, Y']$.
	\end{proof}
	\begin{lemma}\label{l2} Let $G$ be a subgraph of $K_{|X|,|Y|}$, where $|X|=m\geq 9$, and $|Y|=n\geq 9$.  If $\Delta (G_X)\geq 7$, then either $K_{2,2} \subseteq G$ or $K_{4,4} \subseteq \overline{G}$.
	\end{lemma}
	\begin{proof}
		W.l.g let $N_G(x)=Y'$, where  $|Y'|=7$. Suppose that  $K_{2,2} \nsubseteq G$.  Since  $|X|\geq 9$ and $|Y'|=7$, if $|N_G(x')\cap Y'|= 0$ for some $x'\in X\setminus \{x\}$ then by pigeon-hole principle one can check that $K_{4,4} \subseteq \overline{G}[X\setminus\{x\}, Y']$. So suppose that  $|N_G(x')\cap Y'|= 1$ for each $x'\in X\setminus \{x\}$. Therefore as $|X|\geq 9$ and $|Y'|=7$ it is easy to say that there exists two vertices of $X\setminus\{x\}$ say $x',x''$, such that $N_G(x')\cap Y'= N_G(x'')\cap Y' $. Let $N_G(x')\cap Y'= N_G(x'')\cap Y'=\{y'\}$.  Hence  by pigeon-hole principle one can check that $K_{4,4} \subseteq \overline{G}[X\setminus\{x\}, Y'\setminus \{y'\}]$.  Hence lemma holds.
	\end{proof}
	\section{\bf Proof of the main results}
In this section, we prove   Theorem \ref{M.th}.   To simplify the comprehension, let us split the proof of Theorem \ref{M.th} into small parts. We begin with  the following result: 
	\begin{theorem}\label{th2}
		$BR_5(K_{2,2}, K_{4,4})=26$.
	\end{theorem}
	\begin{proof}
		Suppose  that $(X=\{x_1,x_2,x_{3},x_5\},Y=\{y_1,y_2,\ldots ,y_{25}\})$ be the partition  sets of  $K=K_{5,25}$ and $G$ be a subgraph of $K$, such that for each $x\in X$, $N_G(x)$ is defined as follow:
		\begin{itemize}
		 	
			\item{\bf (A1):} $N_G(x_1)=\{y_1,y_2,\ldots, y_7\}$.	
			\item{\bf (A2):} $N_G(x_2)=\{y_1,y_8, y_9\ldots, y_{13}\}$.
			\item{\bf (A3):} $N_G(x_3)=\{y_2,y_8, y_{14},y_{15}\ldots, y_{18}\}$.
			\item{\bf (A4):} $N_G(x_4)=\{y_3,y_9, y_{14}, y_{19},\ldots, y_{22}\}$.
			\item{\bf (A5):} $N_G(x_5)=\{y_4,y_{10}, y_{15}, y_{19}, y_{23}, y_{24}, y_{25}\}$.
		\end{itemize}
		Now for each $i,j \in \{1,\ldots,5\}$, by $(Ai)$ and $(Aj)$, it can be said that  $|N_G(x_i)\cap N_G(x_j)|=1$, which means that $K_{2,2}\nsubseteq G$. Also, for each $i\in \{1,2\ldots,5\}$, it is easy to check that $|\cup_{j=1, j\neq i}^{j=5} N_G(x_j)|=22$. Therefore,  $K_{4,4}\nsubseteq \overline{G}[X\setminus\{x_i\}, Y]$ for each $x_i\in X$, which means that  $BR_5(K_{2,2}, K_{4,4})\geq 26$. 
		
		Now assume that $(X=\{x_1,x_2,x_{3},x_5\},Y=\{y_1,y_2,\ldots ,y_{26}\})$ be the partition  sets of  $K_{5,26}$ and let $G$ be a subgraph of $K_{5,26}$, such that $K_{2,2} \nsubseteq G$. Consider $\Delta=\Delta (G_X)$. Since $K_{2,2} \nsubseteq G$, if $\Delta\geq 8$, then by Lemma \ref{l1}, the proof is complete. Also, since $|Y|=26$, if $\Delta\leq 5$, then it is easy to check that $K_{4,4}\subseteq \overline{G}$. Therefore let $\Delta\in \{6,7\}$. Now we verify  the next claim.
		\begin{claim}\label{c1}
			If	$ \Delta=7$, then we have $K_{4,4}\subseteq \overline{G}$.
		\end{claim}	
		\begin{proof}[Proof of Claim \ref{c1}]
			W.l.g let $\Delta=\deg_G(x_1)$ and $ N_G(x_1)=Y_1=\{y_1,\ldots,y_7\}$. Since $K_{2,2} \nsubseteq G$, so $|N_G(x_i)\cap Y_1|\leq 1$ and $N_G(x_i)\cap Y_1\neq N_G(x_{j})\cap Y_1$ for any  $i,j\in \{2,3,4,5\}$, otherwise $K_{4,4}\subseteq \overline{G}[X,Y_1]$. Suppose that there exists a vertex of $X\setminus \{x_1\}$, so that $|N_{G}(x)|= 6$, and  w.l.g suppose that $x=x_2$ and $ N_G(x_2)=Y_2=\{y_1,y_8\ldots,y_{12}\}$. Therefore, since $K_{2,2} \nsubseteq G$, we have $|N_G(x_i)\cap Y_2\setminus\{y_1\}|\leq 1$, also one can checked that $|N_G(x_i)\cap Y_2\setminus\{y_1\}|= 1$  for at least two   $i\in \{3,4,5\}$, otherwise $K_{4,4}\subseteq\overline{G}[X\setminus\{x_2\},Y_2\setminus\{y_1\}]$. W.l.g assume that  $|N_G(x_i)\cap Y_2\setminus\{y_1\}|= 1$ for $i=3,4$. So, as $ \Delta=7$, $|N_G(x_i)\cap (Y_1\cup Y_2)|= 2$, one can say that $|N_G(x_i)\cap Y\setminus (Y_1\cup Y_2)|\leq  5$, therefore  $|\cup_{j=1}^{j=4} N_G(x_j)|\leq 22$, which means that $K_{4,4}\subseteq \overline{G}$.
			
			Now let $|N_{G}(x)|=7$ for each $x\in X$. Suppose that $N_G(x_i)=Y_i$ for each $i\in \{1,2,\ldots,5\}$. Hence Since $K_{2,2} \nsubseteq G$, one can say that  $|N_G(x_i)\cap Y_j|= 1$ and $N_G(x_i)\cap Y_j\neq N_G(x_{l})\cap Y_j$ for each  $i,l,j\in \{1,2,3,4,5\}$ where $i\neq j\neq l\neq i$, otherwise $K_{4,4}\subseteq \overline{G}[X,Y]$. Therefore it is easy to say that $|\cup_{j=1, j\neq i}^{j=5} N_G(x_j)|=22$. Hence as $|Y|=26$  for each $x\in X$ we have $K_{4,4}\subseteq \overline{G}[X\setminus\{x\}, Y]$.  So, claim  holds.
		\end{proof}
		
		Hence, we may assume that $\Delta=6$. W.l.g let  $N_G(x_1)=Y_1$. Since $K_{2,2} \nsubseteq G$, thus $|N_G(x_i)\cap Y_1|\leq 1$ for each $i\in \{2,3,4,5\}$.  Also, we may suppose that $|N_G(x_i)\cap Y_1|=1$ for at least three vertices of $X\setminus \{x_1\}$, otherwise $K_{4,4}\subseteq \overline{G}[X,Y_1]$. Now, w.l.g let $X_1=\{ x\in X\setminus \{x_1\}, ~~ |N_G(x_i)\cap Y_1|=1\}$. As $|Y|=26$, $|X_1|\geq 3$, and $ \Delta= 6$, one can check that $|\cup_{x\in X_1\cup\{x_1\}} N_G(x)|\leq 21$, which means that $K_{4,4}\subseteq \overline{G}$. Hence we have  $BR_5(K_{2,2}, K_{4,4})= 26$.
		
	\end{proof}
	In the following theorem, we compute the size of $BR_m(K_{2,2}, K_{4,4})$ for $m=6,7$.
	\begin{theorem}\label{th3}
		$BR_6(K_{2,2}, K_{4,4})=BR_7(K_{2,2}, K_{4,4})=22$.
	\end{theorem}
	\begin{proof}
		It suffices to show that :
		\begin{itemize}
			\item{\bf (I):}  $K_{7,21} \rightarrow (K_{2,2},K_{4,4})$.
			\item{\bf (II):} $BR_6(K_{2,2}, K_{4,4})\leq 22$.
		\end{itemize}
		We begin with $(I)$. Let $(X=\{x_1,\ldots, x_7\},Y=\{y_1,\ldots ,y_{21}\})$ be the partition  sets of  $K=K_{7,21}$ and let $G$ be a subgraph of $ K$, such that for each $x\in X$ we define $N_G(x)$ as follow:
		\begin{itemize}
			
			\item{\bf (D1):} $N_G(x_1)=\{y_1,y_2,y_3, y_4,y_5, y_6\}$.	
			\item{\bf (D2):} $N_G(x_2)=\{y_1,y_7, y_8,y_9,y_{10}, y_{11}\}$.
			\item{\bf (D3):} $N_G(x_3)=\{y_2,y_7, y_{12},y_{13}, y_{14}, y_{15}\}$.
			\item{\bf (D4):} $N_G(x_4)=\{y_3,y_8, y_{12}, y_{16},y_{17}, y_{18}\}$.
			\item{\bf (D5):} $N_G(x_5)=\{y_4,y_{9}, y_{13}, y_{16}, y_{19}, y_{20}\}$.
			\item{\bf (D6):} $N_G(x_6)=\{y_5,y_{10}, y_{14}, y_{17}, y_{19}, y_{21}\}$.
			\item{\bf (D7):} $N_G(x_7)=\{y_6,y_{11}, y_{15}, y_{18}, y_{20}, y_{21}\}$.
		\end{itemize}
		Now, for each $i,j \in \{1,2,\ldots,7\}$ by $(Di)$ and $(Dj)$ it can be said that:
		\begin{itemize}
			
			\item{\bf (E1):}  $|N_G(x_i)\cap N_G(x_j)|=1$ for each $i,j\in \{1,2,\ldots,7\}$.
			\item{\bf (E2):} $|\cup_{i=1}^{i=4} N_G(x_{j_i})|=18$ for each $j_1,j_2,j_3,j_4\in \{1,\ldots,7\}$.
		\end{itemize}
		
		By $(E1)$,  one can say that $K_{2,2}\nsubseteq G$. Also by $(E2)$, it is easy to say that  $K_{4,4}\nsubseteq \overline{G}$ which means that  $K_{7,21} \rightarrow (K_{2,2},K_{4,4})$, that is the part $(I)$ is correct.

		Now we show that $(II)$ is established, that is, we show that for any  subgraph  of $K_{6,22}$ say $G$, either $K_{2,2}\subseteq G$ or $K_{4,4}\subseteq \overline{G}$. Let $(X=\{x_1,x_2,\ldots,x_6\},Y=\{y_1,y_2,\ldots ,y_{22}\})$ be the partition  sets of  $K=K_{6,22}$ and $G$ be a subgraph of $K$, where $K_{2,2} \nsubseteq G$. We show that $K_{4,4}\subseteq \overline{G}$.  Consider $\Delta=\Delta (G_X)$. Since $K_{2,2} \nsubseteq G$, by Lemma \ref{l1},  $\Delta\leq 7$.  Now we verify  the following claim.
		\begin{claim}\label{c2}
		If	$ \Delta= 7$, then $K_{4,4}\subseteq \overline{G}$.
		\end{claim}	
		\begin{proof}[Proof of Claim \ref{c2}]
		 W.l.g let $|N_G(x_1)=Y_1|=7$. Since $K_{2,2} \nsubseteq G$ for each  $i,j\in \{2,\ldots,6\}$, we have $|N_G(x_i)\cap Y_1|\leq 1$ and $N_G(x_i)\cap Y_1\neq N_G(x_{j})\cap Y_1$ , otherwise since $|Y_1|=7$, it is easy to check that $K_{4,4}\subseteq \overline{G}[X,Y_1]$. Therefore  as $|X|=6$ and $|Y_1|=7$, for each $x\neq x_1$ it can be said that  $K_{4,3}\subseteq \overline{G}[X\setminus \{x_1,x\},Y_1]$. So, as $|Y|=22$, if there exists a vertex of $Y\setminus Y_1$, so that $|N_{\overline{G}}(y)\cap (X\setminus\{x_1\})|\geq 4$, then  $K_{4,4}\subseteq \overline{G}[X\setminus \{x_1\}, Y_1\cup \{y\}]$. Hence assume that  $|N_G(y)\cap (X\setminus\{x_1\}) |\geq 2$ for each $y\in Y\setminus Y_1$, that is $|E(G[X, Y\setminus Y_1])|\geq 30$. Therefore by pigeon-hole principle, one can check that there exist at least one vertices of $X\setminus\{x_1\}$ say $x_2$, such that $|N_G(x_2)\cap (Y\setminus Y_1)|\geq 5$. Suppose that $N_G(x_2)\cap (Y\setminus Y_1)=Y_2$. So as $K_{2,2} \nsubseteq G$,  $|N_G(x_i)\cap Y_2|\leq 1$  for each  $i\in \{3,4,5,6\}$. Therefore as $|Y_2|\geq 5$,  we have $K_{4,1}\subseteq \overline{G}[X\setminus\{x_1,x_2\}, Y_2]$. Hence as $K_{4,3}\subseteq \overline{G}[X\setminus \{x_1,x_2\},Y_1]$, it can be said that, $K_{4,4}\subseteq \overline{G}[X\setminus\{x_1,x_2\}, Y_1\cup Y_2]$. Hence claim holds.
		\end{proof}
		Therefore, by Claim \ref{c2} assume that	$ \Delta\leq 6$. Let $\Delta=5$ and w.l.g suppose that $\Delta=|N_G(x_1)|$. Since $K_{2,2} \nsubseteq G$, thus $|N_G(x_i)\cap  N_G(x_1)|= 1$ for at least three vertices of $X\setminus\{x_1\}$ say $x_2,x_3,x_4$, otherwise $K_{4,4}\subseteq \overline{G}[X\setminus \{x_1\}, Y_1]$. Hence, as $\Delta=5$, one can check that $|\cup_{i=1}^{i=4}Y_i|\leq 17$, where $Y_i=N_G(x_i)$. So, since $|Y|=22$, we have  $K_{4,4}\subseteq \overline{G}$. 
		
		So let $\Delta=6$. W.l.g let $\Delta=\deg_G(x_1)$ and $ N_G(x_1)=Y_1$. Since $K_{2,2} \nsubseteq G$, thus $|N_G(x_i)\cap Y_1|\leq 1$ for each $i\in \{2,3,4,5,6\}$. Now we verify the following claim.
		\begin{claim}\label{c3}
			$|N_G(x_i)\cap Y_1|=1$ and $N_G(x_i)\cap Y_1\neq N_G(x_j)\cap Y_1$ for each $i,j\in\{2,3,4,5,6\}$.
		\end{claim}	
		\begin{proof}[Proof of Claim \ref{c3}]
			By contradiction,  w.l.g assume that  $|N_G(x_2)\cap Y_1|=0$. Therefore,  $|N_G(x_i)\cap Y_1 |= 1$ and $N_G(x_i)\cap Y_i\neq N_G(x_{j})\cap Y_i$ for each $i,j\in\{3,4,5,6\}$, otherwise $K_{4,4}\subseteq \overline{G}[X,Y_1]$. Now  as $|X|=|Y_1|=6$, it can be said that $K_{4,2}\subseteq \overline{G}[X\setminus \{x_1,x_2\},Y_1]$. If $|N_G(x_2)|=6$ then $K_{4,2}\subseteq \overline{G}[X\setminus \{x_1,x_2\}, N_G(x_2)]$, hence $K_{4,4}\subseteq \overline{G}[X\setminus \{x_1,x_2\}, Y_1\cup N_G(x_2)]$. So suppose that $|N_G(x_2)|\leq 5$. If $|N_G(x_2)|\leq 3$ then as $|Y|=22$ it is clear that $K_{4,4}\subseteq \overline{G}$. Hence  we may suppose that $4\leq |N_G(x_2)|\leq 5$. If $|N_G(x_2)|= 4$, then there exist at  least two vertices of $X\setminus\{x_1,x_2\}$ say $x',x''$, such that $|N_G(x_2)\cap N_G(x)|=1$ for each $x\in \{x',x''\}$, otherwise $K_{4,4}\subseteq \overline{G}[X\setminus\{x_2\}, N_G(x_2)]$. W.l.g assume that $x'=x_3,x''=x_4$. Hence as $\Delta=6$, $ |N_G(x_2)|=4$, and $|N_G(x_i)\cap N_G(x_j)|=1$ for each $i\in \{1,2\}$ and $j\in \{3,4\}$, one can check that $|\cup_{j=1}^{j=4} N_G(x_j)|\leq 18$, which means that $K_{4,4}\subseteq \overline{G}$. So suppose that $|N_G(x_2)|= 5$, hence for at least  three vertices of $X\setminus\{x_1,x_2\}$ say $\{x_3,x_4,x_5\}$, we have $|N_G(x_2)\cap N_G(x_i)|=1$, otherwise $K_{4,4}\subseteq \overline{G}$. If $|N_G(x_i)|\leq 5$ for at least one vertex of $\{x_3,x_4,x_5\}$, then as $\Delta=6$, $|N_G(x_2)|=5$, and $|N_G(x_i)\cap N_G(x_j)|=1$ for each $i\in \{1,2\}$ and each $j\in \{3,4\}$, one can say that $|\cup_{j=1}^{j=4} N_G(x_j)|\leq 18$, which means that $K_{4,4}\subseteq \overline{G}$. So suppose that $|N_G(x_i)|= 6$, for each $x\in \{x_3,x_4, x_5\}$. Consider $N_G(x_3)$, hence there is  at  least  one vertex of $\{x_4,x_5\}$ say $x$, such that $|N_G(x_3)\cap N_G(x_2)\cap (Y\setminus (Y_1\cup Y_2))|=1$ where $Y_i=N_G(y_i)$, otherwise $K_{4,4}\subseteq \overline{G}$. Therefore  $|N_G(x_3)\cap (Y\setminus (Y_1\cup Y_2\cup Y_3))|=3$, which means that $|\cup_{j=1}^{j=4} N_G(x_j)|= 18$. Therefore, it is easy to say that  $K_{4,4}\subseteq \overline{G}[X\setminus\{x_5,x_6\}, (Y\setminus (Y_1\cup Y_2\cup Y_3\cup Y_4)) ]$. For the case that there exists $i,j\in\{2,3,\ldots,6\}$ such that $N_G(x_i)\cap Y_1= N_G(x_j)\cap Y_1$ the proof is same.
		\end{proof}
		
		Hence by Claim \ref{c3},  w.l.g we may suppose that $Y_1=\{y_1,\ldots,y_6\}$ and let $x_iy_{i-1}\in E(G)$ for $i=2,\ldots,6$.   Now we verify the following claim.
		\begin{claim}\label{c4}
			$|N_G(x_i)|=6=\Delta$  for each $i\in\{2,3,4,5,6\}$.
		\end{claim}	
		\begin{proof}[Proof of Claim \ref{c4}]
			By contradiction,  assume that  $|N_G(x_2)|\leq 5$, that is  $|N_G(x_2)\cap (Y\setminus Y_1)|\leq 4$. W.l.g assume that $N_G(x_2)\cap (Y\setminus Y_1)=Y_2$. Since $|X|=6$, one can say that there is at least two vertices of $X\setminus \{x_1,x_2\}$ say $x_3,x_4$, such that $|N_G(x_i)\cap Y_2|=1$, for $i=2,3$, otherwise $K_{4,4}\subseteq \overline{G}[X\setminus \{x_2\}, Y_2]$. Hence as $\Delta=6$, one can check that $|\cup_{i=1}^{i=4}Y_i|\leq 18$, where $Y_i=N_G(x_i)$. So as  and $|Y|=22$, we have  $K_{4,4}\subseteq \overline{G}[ \{x_1,x_2,x_3,x_4\}, Y\setminus \cup_{i=1}^{i=4}Y_i]$
		\end{proof}
		Therefore by Claims \ref{c3} and \ref{c4}, it can be said that $|\cup_{i=1}^{i=6}Y_i|= 21$, that is there exists one vertex of $Y$, say $ y_{22}$, such that $K_{6,1}\subseteq \overline{G}[X, \{y_{22}\}]$. Suppose that $Y_i=N_G(x_i)$ for $i=1,2$, and w.l.g assume that $Y_1=\{y_1,\ldots, y_6\}$ and $Y_2=\{y_1,y_7,\ldots y_{11}\}$. Hence  by Claims \ref{c3}  and \ref{c4} it is easy to say that  there exists one vertex of $Y_1\setminus \{y_1\}$ say $y'$ and one vertex of $Y_2\setminus \{y_1\}$ say $y''$ such that $K_{4,2}\subseteq \overline{G}[\{x_3,\ldots,x_6\}, \{y',y''\}]$. Hence one can check that $K_{4,4}\subseteq \overline{G}[\{x_3,\ldots,x_6\}, \{y_1,y',y'',y_{22}\}]$.
		
		Thus, for any  subgraph of  $K_{6,22}$ say $G$, either $K_{2,2}\subseteq G$ or $K_{4,4}\subseteq \overline{G}$.	Hence, $BR_6(K_{2,2}, K_{4,4}) \leq 22$ and so by $(I)$, $BR_6(K_{2,2}, K_{4,4})= 22$. This also implies that for any  subgraph  of $K_{7,22}$ say $G$, either $K_{2,2}\subseteq G$ or $K_{4,4}\subseteq \overline{G}$. Therefore,  by $(I)$, $BR_7(K_{2,2}, K_{4,4})= 22$. Hence theorem holds.
		\
		
	\end{proof}
	%%%%%%%%%%%%%%%%%%%%%%%%%%%%%%%%%%%%%%%%%%%%%%%%%%%%%%%%%%%%%%%%%%%%%%%%%%%%%%%%%%%%%%%%%%%%%%%%%%%%%%%%%%%%%%%
	
	Let $G$ be a subgraph of  $K_{|X|, |Y|}=K_{m,n}$  where $X=\{x_1,\ldots,x_m\}$ and $Y=\{y_1,\ldots,y_n\}$ are the partition of $K_{|X|, |Y|}$. For any subgraph $G$ of $K_{m,n}$ suppose that $A[G]=[a_{ij}]$ be an $m\times n$ matrix, where for each $i\in [m]$ and each $j\in [n]$, $a_{ij}=1$ if the edges $x_iy_j\in E(G)$, and  $a_{ij}=0$ if the edges $x_iy_j\in E(\overline{G})$. The matrix $A[G]=[a_{ij}]$ represents a $2$-coloured $K_{m,n}$. In the next two theorems, we find the value of $BR_8(K_{2,2}, K_{4,4})$,   by considering a particular $2$-coloured of $K_{8,15}$ and a $2$-coloured  of $K_{8,16}$.	 
	\begin{theorem}\label{th4}
		$K_{8,15}\rightarrow (K_{2,2}, K_{4,4})$.
	\end{theorem}
	\begin{proof}
		\medskip
		Let $G$ be a subgraph of $K_{8,15}$ such that  $A[G]$ is shown in the following matrix:
		
		\[A[G]= A_{8\times 15}=\begin{bmatrix}
			
			1 & 1 & 1 & 1 & 0 & 0 & 0 & 0 \\
			1 & 0 & 0 & 0 & 1 & 0 & 1 & 0 \\
			1 & 0 & 0 & 0 & 0 & 1 & 0 & 1 \\
			1 & 0 & 0 & 0 & 0 & 0 & 0 & 0 \\
			0 & 1 & 0 & 0 & 1 & 0 & 0 & 0 \\
			0 & 1 & 0 & 0 & 0 & 1 & 1 & 0 \\	
			0 & 1 & 0 & 0 & 0 & 0 & 0 & 1 \\ 		
			0 & 0 & 1 & 0 & 1 & 0 & 0 & 0\\		
			0 & 0 & 1 & 0 & 0 & 1 & 0 & 0 \\ 		
			0 & 0 & 1 & 0 & 0 & 0 & 1 & 0 \\ 					
			0 & 0 & 0 & 1 & 1 & 0 & 0 & 1 \\ 		
			0 & 0 & 0 & 1 & 0 & 1 & 0 & 0 \\ 		
			0 & 0 & 0 & 1 & 0 & 0 & 1 & 0 \\
			0 & 0 & 0 & 0 & 1 & 1 & 0 & 0 \\ 		
			0 & 0 & 0 & 0 & 0 & 0 & 1 & 1\\
		\end{bmatrix}
		\]
		Set $X_1=\{x_1,x_2,x_3,x_4\}$ and $X_2=X\setminus X_1$.	Therefore by matrix $A[G]$  it can be said that the following items are true:
		
		\begin{itemize}
			\item {\bf (P1):}   $|N_G(x_i)\cap N_G(x_j)|=1$ for each $i,j\in \{1,\ldots,8\}$.	
			\item {\bf (P2):}   $|N_G(x_i)|=4$ for each $x\in X_1$.
			\item {\bf (P3):}   $|N_G(x_i)|=5$ for each $x\in X_2$.	 
			\item{\bf  (P4):}   $|\cup_{i=1}^{i=4} N_G(x_i)|=13$.
			\item{\bf  (P5):}   $|\cup_{i=5}^{i=8} N_G(x_i)|=14$.

		\end{itemize}
 Therefore by $(P1)$ we have $G$ is $K_{2,2}$-free, also	by $(P4)$ and $(P6)$, for $i=1,2$ it can be said that $K_{4,4}\nsubseteq \overline{G}[X_i, Y]$.
		Also  by $A[G]$ it can be said that the following items are true:
		\begin{itemize}
			
			\item{\bf  (M1):}   $|N_G(x_i)\cup N_G(x_j)|=7$ for each $x_i,x_j\in X_1$.
			\item{\bf  (M2):}   $|N_G(x_i)\cup N_G(x_j)\cup N_G(x_l)|=10$ for each $x_i,x_j,x_l\in X_1$.
			\item{\bf  (M3):}   $|N_G(x_i)\cup N_G(x_j)|=9$ for each $x_i,x_j\in X_2$.
			\item{\bf  (M4):}   $|N_G(x_i)\cup N_G(x_j)\cup N_G(x_l)|=12$ for each $x_i,x_j,x_l\in X_2$.
		\end{itemize}
	Now we verify the following claim.
		\begin{claim}\label{c5}
			$\overline{G}$ is $K_{4,4}$-free. 
		\end{claim}	
		\begin{proof}[Proof of Claim \ref{c5}] 
			By contradiction, suppose that $K$ is a copy of $K_{4,4}$ in $\overline{G}$, also assume that $V(K)\cap X=\{w_1,w_2,w_3,w_4\}=W$ and   $V(K)\cap Y=\{w'_1,w'_2,w'_3,w'_4\}=W'$.  Since $K_{4,4}\subseteq \overline{G}[W,W']$, then by $(P4)$ and $(M4)$, one can say that $1\leq |W\cap X_2|\leq 2$. Assume that $|W\cap X_i|=2$. W.l.g let $w_1,w_2\in X_1$ and $w_3,w_4\in X_2$. Hence, by $(M1)$, we have $|N_G(w_1)\cup N_G(w_2)|=7$. Now consider $w_3,w_4$, as  $|N_G(w_i)\cap N_G(w_j)|=1$ for each $i,j\in \{1,\ldots,8\}$, and $y_1\notin N_G(x)$ for each $x\in X_2$ one can say that 	$|N_G(w_1)\cup N_G(w_2)\cup N_G(w_i)|=10$ for each $i\in\{3,4\}$. Also, as $|N_G(w_i)\cap N_G(w_j)|=1$, it  is easy to check that  $|N_G(w_1)\cup N_G(w_2)\cup N_G(w_3)\cup N_G(w_4)|\geq 12$, which means that $K_{4,4}\nsubseteq \overline{G}[W,Y]$, a contradiction. So assume that $|W\cap X_1|=3$, and w.l.g let $w_4\in X_2$. Hence, by $(P1)$, $(P3)$, and  $(M2)$,  one can say that 	$|N_G(w_1)\cup N_G(w_2)\cup N_G(w_3)\cup N_G(w_4)|=12$, which means that $K_{4,4}\nsubseteq \overline{G}[W,Y]$, a contradiction again. Hence claim holds.	 
		\end{proof}	
		Therefore by $(P1)$ and by Claim \ref{c5}, we have the proof of the theorem is complete.	
	\end{proof}
To proving the next theorem, we need to establish the following lemma.
	\begin{lemma}\label{l3}{(\cite{collins2016zarankiewicz})}The following results on  $z((m,n), K_{t,t})$ are true: 	
	\begin{itemize}	
		\item[$\bullet$] $z((6,9), K_{2,2})\leq  21$.
		\item[$\bullet$] $z((6,12), K_{2,2})\leq  26$.
		\item[$\bullet$] $z((5,6), K_{2,2})\leq  14$.
		\item[$\bullet$] $z((7,9), K_{2,2})\leq  24$.
		\item[$\bullet$] $z((7,12), K_{2,2})\leq 28$.
			\item[$\bullet$] $z((7,16), K_{2,2})\leq 34$.
	\end{itemize}
\end{lemma}
	\begin{proof}
		By using the bounds in Table $5$ of \cite{collins2016zarankiewicz} and Table C.4 of \cite{collins2015bipartite}, the  proposition holds.
	\end{proof}

	\begin{theorem}\label{th5}
		$BR_8(K_{2,2}, K_{4,4})=16$.
	\end{theorem}
	\begin{proof}
		 Let $G$ be any subgraph of $K_{8,16}$. Since  $|E(K_{8,16})|=128$, then we may assume that $|E(G)|=38$ and $|E(\overline{G})|= 90$, otherwise by  Lemma \ref{l0} as	$z((8,16), K_{2,2})\leq  38$ and $z((8,16), K_{4,4})\leq  90$ it can be said that either $K_{2,2}\subseteq G$ or $K_{4,4}\subseteq \overline{G}$, that is the proof is complete. Now let  $|E(G)|=38$,  $|E(\overline{G})|= 90$ and w.l.g assume that $K_{2,2}\nsubseteq G$. Since $z((7,16), K_{2,2})\leq  34$, if there exist a vertex of $G$ say $x$, such that $|N_G(x)|\leq 3$, then it is easy to say that $|E(G\setminus \{x\})|\geq 35$. Therefore, we have $K_{2,2}\nsubseteq G$, a contradiction. Hence, $\delta(G)=4$. Define $X'$ as follow:
		\[X'=\{x\in X, ~~\deg_G(x)= 4\}\]
		Since $|E(G)|=38$ and $|X|=8$, it is clear that  $|X'|\geq 2$. By considering $X'$,  we verify the following claim.
		\begin{claim}\label{c6}
			If there exist two members of $X'$, say $x,x'$ such that $|N_G(x)\cap N_G(x')|=1$, then $K_{4,4}\subseteq \overline{G}$. 
		\end{claim}	
		\begin{proof}[Proof of Claim \ref{c6}]
			W.l.g let $x=x_1,x'=x_2$,  $N_G(x_1)=Y_1=\{y_1,y_2,y_3,y_4\}$ and $N_G(x_2)=Y_2=\{y_1,y_5,y_6,y_7\}$. Now set $Y'=Y_1\cup Y_2$. Hence we have the following fact.
			\begin{itemize}
				\item {\bf (F1):}  There is at least one member of $X_1=X\setminus\{x_1,x_2\}$ say $x$, such that $|N_G(x)\cap Y'|=2$. 
			\end{itemize}	 
		 To proving $(F1)$, by contrary, let $|N_G(x)\cap Y'|\leq 1$ for each $x\in X_1$. Set $X_1=X\setminus\{x_1,x_2\}$ and $Y'=Y\setminus( Y_1\cup Y_2)$, therefore we have $|E(G[X_1,Y'])|\geq 25$. Hence by Lemma \ref{l3} as $z((6,9), K_{2,2})\leq 24$, we have $K_{2,2}\subseteq G$, a contradiction. Hence The Fact $(F1)$ is correct.
			
			Therefore, by $(F1)$ w.l.g let $|N_G(x_3)\cap Y'|=2$ and   $Y_3=\{y_2,y_5, y_8, y_9\}\subseteq N_G(x_3)$. If $x_3\in X'$, then by an argument similar to the proof of $(F1)$, one can check that there exists at least one member of $X\setminus\{x_1,x_2,x_3\}$ say $x$, such that $|N_G(x)|\leq 5$ and $|N_G(x)\cap (Y'\cup Y_3)|\geq 2$. Therefore, $|N_G(x)\cap Y\setminus (Y'\cup Y_3)|\leq 3$, hence it is easy to say that $K_{4,4}\subseteq \overline{G}[\{x_1,x_2,x_3,x\}, Y\setminus (Y'\cup Y_3) ]$. So, suppose that  	$|N_G(x_3)|\geq 5$. Now we verify the following two cases.
			
			{\bf Case 1:} $|N_G(x_3)|= 5$.	In this case,  we have   the following fact.
			\begin{itemize}
				\item {\bf (F2):}  There exists one member of $X\setminus\{x_1,x_2, x_3\}$ say $x$, such that $|N_G(x)|=5$ and $|N_G(x)\cap (Y'\cup Y_3)|=3$. 
			\end{itemize}	 
		For prove $(F2)$, by contrary, let $|N_G(x)\cap (Y'\cup Y_3)|\leq 2$ for each $X\setminus\{x_1,x_2, x_3\}$. Set $X_1=X\setminus\{x_1,x_2,x_3\}$ and $Y''=Y\setminus(Y'\cup Y_3)$, therefore we have $|E(G[X_1,Y''])|\geq 15$. Hence by Lemma \ref{l3} as $z((5,6), K_{2,2})\leq 14$ we have $K_{2,2}\subseteq G$, a contradiction. Hence The Fact $(F2)$ is true.			
			
			Therefore, by $(F2)$ w.l.g let $|N_G(x_4)\cap Y\setminus (Y'\cup Y_3 \cup Y'')|=3$ and let   $Y_4=\{y',y'', y'', y_{10}, y_{12}\}\subseteq N_G(x_4)$. If $|N_G(x_4)|\leq 5$, then it is easy to say that $K_{4,4}\subseteq \overline{G}[\{x_1,x_2,x_3,x_4\}, Y\setminus (Y'\cup Y_3\cup Y_4) ]$. So suppose that $|N_G(x_4)|=6$	and assume that  $Y_4=\{y',y'', y''', y_{11}, y_{12}, y_{13}\}\subseteq N_G(x_4)$. Now set $X_2=\{x_5,x_6,x_7,x_8\}$ and $Y'''=\{y_{14}, y_{15}, y_{16}\}$, therefore one can check that $|N_G(x)\cap Y'''|\geq 2$ for each $x\in X''$. Otherwise, if there is a vertex of $X''$ say $x''$, such that $|N_G(x'')\cap Y'''|\leq 1$, then it is easy to say that either $|N_G(x'')|=5$ and $|N_G(x'')\cap Y'\cup Y_3|=3$ or $|N_G(x'')|=4$ and $|N_G(x'')\cap Y'\cup Y_3|=2$, in any case the Fact (F2) is true by setting $x=x''$.  Therefore as $|X_2|=4$, $|Y'''|=3$ and $|N_G(x)\cap Y'''|\geq 2$ for each $x\in X''$, it is easy to checked that  $K_{2,2}\subseteq G[X_2,Y''']$, a contradiction.
			
			{\bf Case 2:} $|N_G(x_3)|= 6$.  W.l.g assume that $N_G(x_3)=Y_3=\{y_2,y_5, y_8, y_9, y_{10}, y_{11}\}$.  Now set $X_2=\{x_4, x_5,x_6,x_7,x_8\}$ and $Y'''=\{y_{12},\ldots, y_{16}\}$. If there exists a vertex of $X_2$ say $x$, such that $|N_G(x)\cap Y'''|\leq 2$, then it is easy to say that either $|N_G(x)|=5$ and $|N_G(x)\cap Y'|=2$ or $|N_G(x)|=4$ and $|N_G(x)\cap Y'|\geq 1$, in any case the proof is complete by Case 1. Therefore, let $|N_G(x)\cap Y'''|\geq 3$ for at lest three vertices of $x\in X''$. Now as  $|Y'''|=5$, its easy to say that $K_{2,2}\subseteq G[X'',Y''']$, a contradiction.
			
			Therefore by Cases 1,2, the proof of claim is complete.
		\end{proof}
		Now by Claim \ref{c6}, we have the following claim.
		\begin{claim}\label{c7}
			$ \Delta=5$.
		\end{claim}	
		\begin{proof}[Proof of Claim \ref{c7}]
			By contradiction,  suppose that $\Delta=6$, therefore $|X'|\geq 3$. W.l.g let $|N_G(x_1)=Y_1|=\Delta$ and $x_2,x_3,x_4\in X'$. Since $K_{2,2} \nsubseteq G$, thus $|N_G(x)\cap Y_1|\leq 1$ for each $x\in X\setminus\{x_1\}$. Also by Claim \ref{c6} we have $|N_G(x)\cap N_G(x')|=0$ for each $x,x'\in X'$. So, w.l.g let $N_G(x_2)=Y_2=\{y_1,y_2,y_3,y_4\}$, $N_G(x_3)=Y_3=\{y_5,y_6,y_7,y_8\}$ and $N_G(x_4)=Y_4=\{y_9,y_{10},y_{11},y_{12}\}$. Set $W=Y\setminus (Y_1\cup Y_2\cup Y_3)= \{y_{13}, y_{14}, y_{15}, y_{16}\}$. As $N_G(x_1)=Y_1$ and $K_{2,2} \nsubseteq G$, we have $|Y_1\cap W|\geq 3$. W.l.g let $W'=\{y_{13}, y_{14}, y_{15}\}\subseteq Y_1\cap W$. Also as $|N_G(x)\cap W|\geq 2$ for each $x\in X\setminus \{x_1,\ldots,x_4\}$, then it is easy to check  that  $K_{2,2} \subseteq G[\{x_5,x_6,\ldots,x_8\}, W]$, a contradiction. Hence claim holds.	
		\end{proof}
		Now by Claim \ref{c7}, and as $|E(G)|=38$, we have $|X'|=2$. W.l.g let $X'=\{x_1,x_2\}$. Also by Claim \ref{c6}, w.l.g suppose that $N_G(x_1)=Y_1=\{y_1,y_2,y_3,y_4\}$, $N_G(x_2)=Y_2=\{y_5,y_6,y_7,y_8\}$.  Set $X_1=X\setminus \{x_1,x_2\}$. Now we verify the next claim.
		\begin{claim}\label{c8}
			For each $x\in X_1$ and  $i=1,2$, we have $|N_G(x)\cap Y_i|=1$.
		\end{claim}	
		\begin{proof}[Proof of Claim \ref{c8}]
			By contrary, w.l.g  suppose that $|N_G(x_3)\cap Y_1|=0$. Since $|X|=8$, and $|N_G(x)|=5$ for each $x\in X_1$,   either if $|N_G(x_3)\cap Y_2|=0$ or there exists a vertex of 	$X_1\setminus\{x_3\}$ such that $|N_G(x)\cap Y_1|=0$, then it is easy to say that $K_{2,2}\subseteq G[X_1, Y\setminus (Y_1\cup Y_2)]$, a contradiction. So let $|N_G(x_3)\cap Y_1|=0$,  $|N_G(x_3)\cap Y_2|=1$ and w.l.g let $N_G(x_3)=Y_3=\{y_5,y_9,y_{10},y_{11}, y_{12}\}$.  Set $X'=X\setminus\{x_1\}$ and $Y'=Y\setminus Y_1$. Since $K_{2,2}\nsubseteq G$, we have  $|N_G(x)\cap Y_1|\leq 1$ for each $x\in X'\setminus\{x_3\}$ and as  $|N_G(x_2)\cap Y_1|=0$ we have $|E(G[X,Y_1])|\leq 9$. Therefore as  $|N_G(x_1)= Y_1|=4$, we  have  $|E(G[X',Y'])|\geq 29$. Hence by Lemma \ref{l3} as $z((7,12), K_{2,2})\leq 28$ we have $K_{2,2}\subseteq G$, a contradiction. Hence claim holds.
		\end{proof}
		Therefore by Claim \ref{c8}, we have  $|N_G(x)\cap Y_i|=1$ for each $x\in X_1$ and each $i=1,2$. If there is a vertex of $Y_1$ or $Y_2$ say $y$, so that  $|N_G(y)\cap X_1|\geq 3$ then as $|Y\setminus(Y_1\cup Y_2)|=8$ and  $|N_G(x)\cap Y\setminus(Y_1\cup Y_2)|= 3$,   by pigeon-hole principle one can check that  $K_{2,2}\subseteq G$, a contradiction. Therefore we may suppose that  $|N_G(y)\cap X_1|\leq 2$ for each $x\in X_1$. Therefore as $|X_1|=6$ and $|Y_1|=4$ there exist two member of $Y_1$ say $y,y'$ such that $|N_G(y)\cap X_1|=|N_G(y')\cap X_1|=1$. W.l.g let $y=y_1$ and $y'=y_2$. By  symmetry as $|X_1|=6$ and $|Y_2|=4$ there exist two member of $Y_2$ say $y'',y'''$ such that $|N_G(y'')\cap X_1|=|N_G(y''')\cap X_1|=1$. W.l.g let $y''=y_5$ and $y'''=y_6$. Noe set  $X'=X\setminus\{x_1, x_2\}$ and $Y'=Y\setminus \{y_1,y_2,y_5,y_6\}$. Since  $|N_G(x_i)|= 4$ for each $i=1,2$, $y_1,y_2\in N_G(x_1)$, $y_5,y_6\in N_G(x_2)$ and as $|N_G(y)\cap X_1|=1$ for each $y\in Y'$, then one can say that   $|E(G[X', \{y_1,y_2,y_5,y_6\}])|= 8$ and $|E(G[\{x_1,x_2\},Y'])|= 4$, therefore as $|E(G)|=38$ one can  said that   $|E(G[X',Y']|= 26$. Hence by Lemma \ref{l3} as $z((6,12), K_{2,2})\leq 25$ we have $K_{2,2}\subseteq G$, a contradiction.
		
		 Hence we have the assumption that  $|E(G)|=38$ and $|E(\overline{G})|= 90$ does not hold. Therefore by  Lemma \ref{l0} as	$z((8,16), K_{2,2})\leq  38$ and $z((8,16), K_{4,4})\leq  90$, it can be said that either $K_{2,2}\subseteq G$ or $K_{4,4}\subseteq \overline{G}$, that is  $BR_8(K_{2,2}, K_{4,4})\leq 16$.
		 Therefore, by Theorem \ref{th4}, we have $BR_8(K_{2,2}, K_{4,4})=16$, and the proof is complete.
		
	\end{proof}
	In the following  theorem,	 we find the value of $BR_9(K_{2,2}, K_{3,3})$.
	\begin{theorem}\label{th6}
		$BR_9(K_{2,2}, K_{3,3})=14$.
	\end{theorem}	
	\begin{proof}
		 Consider $K=K_{9,14}$, let $(X=\{x_1,\ldots,x_9\},Y=\{y_1,y_2,\ldots ,y_{14}\})$ be the partition  sets of  $K$, and  $G$ be a subgraph of $K$ such that $K_{2,2}\nsubseteq G$.  Consider $\Delta=\Delta (G_X)$. By Lemma \ref{l2}, $\Delta\leq 6$. Now, we verify the following claim.
		\begin{claim}\label{c9}
			If	$ \Delta=6$, then $K_{4,4}\subseteq \overline{G}$.
		\end{claim}	
		\begin{proof}[Proof of Claim \ref{c9}]
			W.l.g let  $|N_G(x_1)=Y_1|=6$. Since $K_{2,2} \nsubseteq G$, thus $|N_G(x)\cap Y_1|\leq 1$ for each $x\in X\setminus\{x_1\}$. Therefore,  as $|X|=9$ and $|Y_1|=6$,  if there is a vertices of $X\setminus \{x_1\}$ or a vertex of $Y_1$ say $x, y$ respectively, such that either $|N_{G}(x)\cap Y_1|=0$ or  $|N_{G}(y)\cap X\setminus\{x_1\}|\geq 3$, then  by pigeon-hole principle one can check that $K_{4,4}\subseteq \overline{G}[X\setminus \{x_1\}, Y_1]$. Hence, let  $|N_G(x)\cap Y_1 |=1$ for each $x\in X\setminus \{x_1\}$ and  $|N_{G}(y)\cap X\setminus\{x_1\}|\leq 2$ for each $y\in Y_1$. Now, as $|X|=9$ and $|Y_1|=6$, there exists two member of $Y_1$ say $y_1, y_2$, so that $|N_G(y_i)\cap X\setminus\{x_1\}|= 2$. W.l.g let $N_G(y_1)\cap X\setminus\{x_1\}=\{x_2,x_3\}$ and $N_G(y_2)\cap X\setminus\{x_1\}=\{x_4,x_5\}$. Hence, $K_{4,4}\subseteq \overline{G}[\{x_2,x_3,x_4,x_5\}, Y_1\setminus\{y_1,y_2\}]$, so claim holds.
		\end{proof}
		Since  $|E(K_{9,14})|=126$, then  $38 \leq|E(G)|  \leq 39$ and $87\leq |E(\overline{G})|\leq 88$. Otherwise  by Lemma \ref{l0}, as 	$z((9,14), K_{2,2})\leq  39$ and $z((9,14), K_{4,4})\leq  88$, it can be said that either $K_{2,2}\subseteq G$ or $K_{4,4}\subseteq \overline{G}$.  Now we define $A$ as follow:
		\[ A=\{x\in X,~~\deg_G(x)=5\}.\]   
	By Claim \ref{c9}  and since $K_{2,2}\nsubseteq G$ and  $38 \leq|E(G)|  \leq 39$, it is easy to say that there exist at least two members of $X$ say $x,x'$ such that $\Delta=\deg_G(x)=\deg_G(x')=5$. So $|A|\geq 2$.  Now we verify the following two claims:
			\begin{claim}\label{c12}
			We have  $|A|\leq 3$.
		\end{claim}	
		\begin{proof}[Proof of Claim \ref{c12}] 		
			By contradiction,	w.l.g let $|A|\geq 4$. Now consider $|E(G)|$. Since $38 \leq|E(G)|  \leq 39$, if $|E(G)|=38$, then  as $|X|=9$ one can said that there exist two vertices of $X\setminus A$ say $x',x''$ such that for each $x\in \{x',x''\}$,  $|N_G(x_i)|\leq 3$. Therefore  $|E(G\setminus \{x',x''\})|\geq 32$. Hence  by Lemma \ref{l0}, as 	$z((7,14), K_{2,2})\leq  31$  we have $K_{2,2}\subseteq G[X\setminus\{x_5,x_6\}, Y]$, a contradiction. So suppose that 	$|E(G)|=39$, then  as $|X|=9$  there is one vertex  of $X\setminus A$ say $x$ such that for  $|N_G(x)|\leq 3$.   Therefore we have $|E(G\setminus \{x\})|\geq 36$. Hence  by Lemma \ref{l0}, as 	$z((8,14), K_{2,2})\leq  35$  we have $K_{2,2}\subseteq G[X\setminus\{x\}, Y]$, a contradiction again. Hence claim holds.
		\end{proof}
		\begin{claim}\label{c10}
			For each $x,x'\in A$,  $|N_{G}(x)\cap N_{G}(x') |=1$.  
		\end{claim}	
		\begin{proof}[Proof of Claim \ref{c10}] 
			By contrary, w.l.g let $x_1, x_2\in A$ and  $|N_{G}(x_1)\cap N_{G}(x_2) |=0$, $Y_1= N_G(x_1)=\{y_1,\ldots,y_5\}$ and $N_G(x_2)=Y_2=\{y_6,\ldots,y_{10}\}$. By Claim \ref{c12} assume that $|A|= 3$ and w.l.g let $x_3\in A$, $N_G(x_3)=Y_3$. As   $K_{2,2}\nsubseteq G$, one can say that  $|N_G(x_3)\cap Y\setminus (Y_1\cup Y_2)|\geq 3$. Therefore, since $|X|=9$, $|A|=3$, and  $|E(G)|\leq  39$, we have $|N_G(x)|=4$ for at least five vertices of  $X\setminus A$.  Now, since   $K_{2,2}\nsubseteq G$, one can say that  $|N_G(x)\cap Y\setminus (Y_1\cup Y_2)|=2$ for this  vertices. Hence, as  $|N_G(x_3)\cap Y\setminus (Y_1\cup Y_2)|\geq 3$, it is easy to check that  $K_{2,2}\subseteq G[X\setminus\{x_1,x_2\}, Y\setminus (Y_1,Y_2)]$, a contradiction. So let $|A|=2$. Therefore $|E(G)|= 38$, and $|N_G(x)|=4$ for each $x\in X\setminus A$ and $|N_G(x)\cap Y\setminus (Y_1\cup Y_2)|\geq 2$. Hence as $|Y\setminus (Y_1\cup Y_2)|=4$ and  $|X\setminus A|=7$, by pigeon-hole principle it is easy to say that  $K_{2,2}\subseteq G[X\setminus\{x_1,x_2\}, Y\setminus (Y_1,Y_2)]$, a contradiction again.  So claim holds.
		\end{proof}
	Now,	w.l.g assume that $x_1,x_2\in A$ and let $Y_1=\{y_1,\ldots,y_5\}$, $Y_2=\{y_1,y_6\ldots,y_9\}$, where $Y_i=N_G(x_i)$ for $i=1,2$. Define $A'$ as follow:
		\[ A'=\{x\in X\setminus A ,~~\deg_G(x)=4\}.\] 
By Claim \ref{c12}, as $38\leq |E(G)|\leq  39$, we have $|A'|\geq 5$.  Now we verify the next two claims:		
		\begin{claim}\label{c11}
			If  $|A|= 3$,  then  $y_1\notin N_G(x)$ for each $x\in A\setminus\{x_1,x_2\}$.
		\end{claim}	
		\begin{proof}[Proof of Claim \ref{c11}] 
			W.l.g suppose that $x_3 \in A$ and by contrary let $y_1\in N_G(x_3)$. Hence $|\cup_{i=1}^{i=3}Y_i|=13$. Assume that $\{y_{14}\}= Y\setminus \cup_{i=1}^{i=3}Y_i$. So as   $K_{2,2}\nsubseteq G$,  $y_{14}\in N_G(x)$ and $y_1\notin N_G(x)$ for each $x\in A'$. Also  $|N_G(x)\cap Y_i\setminus\{y_1\}|=1$ for each $x\in A'$ and $i=1,2$. Now, as $|Y_1\setminus\{y_1\}|=4$ and $|A'|\geq 5$, it is clear that   $K_{2,2}\subseteq G[A', \{y,y_{14}\}]$ for some  $y\in Y_1$, a contradiction. 
		\end{proof}

		\begin{claim}\label{c13}
			If  $|A|= 3$, then $K_{4,4}\subseteq \overline{G}$.
		\end{claim}	
		\begin{proof}[Proof of Claim \ref{c13}]
			W.l.g let $ A=\{x_1,x_2,x_3\}$. Also by Claim \ref{c11} and \ref{c10}, w.l.g let  $N_G(x_3)=Y_3=\{y_2, y_6,y_{10}, y_{11}, y_{12}\}$. Therefore  $|\cup_{i=1}^{i=3}Y_i|=12$.	 	Since  $|A'|\geq 5$, w.l.g we may suppose  that $ \{x_4,x_5,x_6,x_7,x_8\}\subseteq A'$. Also one can assume that $|N_G(x)\cap Y'''|=0$ for each $x\in A'$, where $Y'''=\{y_1,y_2,y_6\}$. Otherwise, let $y_1\in N_G(x_4)$( for other case the proof is the same). So, as $K_{2,2}\nsubseteq G$, we have $|N_G(x_4)\cap (Y_3 \setminus \{y_2,y_6\})|=1$ that is $\{y_{13}, y_{14}\}\subseteq N_G(x_4)$. Now w.l.g assume that  $N_G(x_4)=\{y_1,y_{10},y_{13}, y_{14} \}$. Therefore it can be said that $|N_G(x)\cap \{y_1,y_2,y_6,y_{10}\}|=0$ for each $x\in A'\setminus\{x_4\}$, otherwise   $K_{2,2}\nsubseteq G[\{x_4, x\}, \{y, y'\}]$ for some $y\in \{y_1,y_2,y_6,y_{10}\}$ and $y'\in \{y_{13},y_{14}\}$, a contradiction. Therefore,  $K_{4,4}\subseteq \overline{G}[A'\setminus\{x_4\},  \{y_1,y_2,y_6,y_{10}\}]$.
			
			So, let $|N_G(x)\cap Y'''|=0$ for each $x\in A'$. If there is a vertex of $A'$ say $x$, such that $|N_G(x)\cap \{y_{13}, y_{14}\}|=2$, then the proof is same. Hence we may assume that $|N_G(x)\cap \{y_{13}, y_{14}\}|=1$ for each $x\in A'$. Also as $K_{2,2}\nsubseteq G$ for each $x\in A'$, we have $|N_G(x)\cap \{y_3,y_4,y_5\}|=|N_G(x)\cap \{y_7,y_8,y_9\}=|N_G(x)\cap \{y_{10}, y_{11},y_{12}\}=|N_G(x)\cap \{y_{13},y_{14}\}=1$.  Therefore, as $|A'|\geq 5$, by pigeon-hole principle  there exists at leas one member of $\{y_{13}, y_{14}\}$ say $y_{13}$, such that  $|N_G(y_{13})\cap A'|\geq 3$. As  $|N_G(x)\cap \{y_3,y_4,y_5\}|=1$, if $|N_G(y_{13})\cap A'|\geq 4$ then it is easy to check that $K_{2,2}\subseteq G[N_G(y_{13})\cap A', \{y,y_{13}\}]$ for some $y\in \{y_3,y_4,y_5\}$,  a contradiction.  Hence, w.l.g let $\{x_4, x_5, x_{6}\}= N_G(y_{13})\cap A'$, that is $x_7,x_8\in N_G(y_{14})$.
			
			Now, w.l.g assume that $N_G(x_4)= Y_4=\{y_3,y_7,y_{10}, y_{13}\}$, $N_G(x_5)= Y_5=\{y_4,y_8,y_{11}, y_{13}\}$, and $N_G(x_6)= Y_6=\{y_5,y_9,y_{12}, y_{13}\}$. Consider $D=\{y_5,y_9,y_{12}\}$, for $i=7,8$ if there exists a vertex of $D$ say $y$, such that $yx_i\in E(G)$, then $K_{2,2}\subseteq G[\{x_7,x_8\}, \{y,y_{14}\}]$, a contradiction. Also there is a vertex of $D$ say $y$, so that $yx_7,yx_8\in E(\overline{G})$, otherwise $K_{2,2}\subseteq G[\{x_6,x_i\}, D]$ for some $i\in \{7,8\}$, a contradiction again. Therefore w.l.g let $x_7y_5,x_8y_5\in E(\overline{G})$, hence  $K_{4,4}\subseteq \overline{G}[\{x_4,x_5,x_7,x_8\}, Y'''\cup \{y_5\}]$. 	So claim holds.
		\end{proof}
		Now by Claim \ref{c13}, assume that $A=\{x_1,x_2\}$, therefore  $A'=X\setminus\{x_1,x_2\}$.  For $i=1,2$ set $Y'_i=Y_i\setminus \{y_1\}$ and set $Y'=\{y_{10},\ldots,y_{14}\}$.  Now we verify two claims as follow:
		\begin{claim}\label{c14}
			For each $y\in Y'_1\cup Y'_2$,  we have $|N_G(y)\cap A'|\leq 2$.  
		\end{claim}	
		\begin{proof}[Proof of Claim \ref{c14}] 
			By contradiction, let $|N_G(y)\cap A'|\geq 3$ for at least one $y\in Y'_1\cup Y'_2$ and let $X'= N_G(y)\cap A'$. As $|N_G(x)\cap Y'|\geq 2$ and $|X'|\geq 3$, it can be  said that $|N_G(y')\cap X'|\geq 2$ for  one $y'\in Y'$, which means that  $K_{2,2}\nsubseteq G[ X', \{y,y'\}]$, a contradiction.
		\end{proof}	 
		By an argument similar to the proof of Claim \ref{c14},  we can say that  the following claim is established. 
		\begin{claim}\label{c15}
			$|N_G(y_1)\cap A'|\leq 1$.  
		\end{claim}	
		%\begin{proof}[Proof of Claim \ref{c5}] 
		%By contrary, w.l.g assume that $x_3,x_4\in N_G(y_1)$.  Set $Y'=\{y_{10},\ldots,y_{14}\}$, as $K_{2,2}\nsubseteq G$ and $y_1\in N_G(x)$ for each $x\in W$, then  it can be  say that $|N_G(x_i)\cap Y'|=3$ for $i=3,4$. Therefore as $|Y'|=5$, it is easy to say that  $|N_G(x_3)\cap N_G(x_4) \cap Y'|=1$ which means that $K_{2,2}\nsubseteq G[W, \{y_1,y'\}]$ where $y'\in N_G(x_3)\cap N_G(x_4) \cap Y'$, a contradiction.
		%\end{proof}	 
	Now by Claim \ref{c15}, we verify two cases as follow:
		
		{\bf Case 1: $|N_G(y_1)\cap A'|= 0$}. In this case, we verify  the following claim.
		\begin{claim}\label{c16}
			If there exists a vertex of $A'$ say $x$, such that $|N_G(x)\cap Y'|=3$, then  $K_{4,4}\subseteq \overline{G}$.
		\end{claim}	
		\begin{proof}[Proof of Claim \ref{c16}] 
			W.l.g assume that $|N_G(x_3)\cap Y'|=3$. Therefore as $|N_G(x_3)|=4$, one can say that $|N_G(x_3)\cap Y'_i|= 0$ for one $i\in \{1,2\}$. W.l.g assume that $|N_G(x_3)\cap Y'_1|= 0$. Hence, as $|A'\setminus\{x_3\}|=6$, $|Y'_1|=4$,  and $|N_G(x)\cap Y_1'|\leq 1$ for each $x\in A'\setminus \{x_3\}$, then one can say that  there exist two vertices of $Y'_1$ say $y'_1, y'_2$ such that $|N_G(y)\cap A'|= 1$ for each $y\in \{y'_1,y'_2\}$. Also  as $|A'|=7$, $|Y'_2|=4$,  and $|N_G(x)\cap Y_2'|\leq 1$ for each $x\in A'\setminus \{x_3\}$, thus  there exists one vertex of $Y'_2$ say $y'_3$, such that $|N_G(y'_3)\cap A'|= 1$. Set $W=\{y_1,y'_1,y'_2,y'_3\}$. Since $|N_G(y_1)\cap A'|= 0$, we have $ |\cup_{y\in W}(N_G(y)\cap A')|\leq 3$, so since $|A'|=7$, we have $K_{4,4}\subseteq \overline{G}[A',W]$.
		\end{proof}	 
		Therefore by Claim \ref{c16}, we may suppose that  $|N_G(x)\cap Y_i'|=1$ for each $x\in A'$ and each $i\in \{1,2\}$. Hence by Claim \ref{c14}, and using the fact that $|A'|=7$, $|Y_i|=4$, it can be said that there exists one vertex of $Y'_1$ say $y'_1$, and one vertex of $Y'_2$ say $y'_2$, such that $|N_G(y)\cap A'|= 1$ for each $y\in \{y'_1,y'_2\}$. Also as $|N_G(x)\cap Y_i'|=1$ for each $x\in X'$ and each $i\in \{1,2\}$, we have $|N_G(x)\cap Y'|=2$ for each $x\in X'$. Hence by Claim \ref{c14}, using the fact that $|A'|=7$ and $|Y'|=5$, it can be said that there exists one vertex of $Y'$ say $y'_3$, such that $|N_G(y)\cap A'|= 2$. Now set $W=\{y_1,y'_1, y'_2,y'_3\}$. We note that $N_G(y_1)=\{x_1,x_2\}$. Hence assume that $N_G(y'_1)=\{x_1,x'_1\}$, $N_G(y'_2)=\{x_2,x'_2\}$ also let $N_G(y'_3)=\{x'_3,x'_4\}$, where $x'_i\in A'$. If there exists $i,j\in \{1,2,3,4\}$, such that $x'_i=x_j'$, then it can be said that $ |\cup_{y\in W}(N_G(y))|\leq 5$, which means that $K_{4,4}\subseteq \overline{G}[X,W]$. So suppose that  $x'_i\neq x_j'$ for each $i,j\in \{1,2,3,4\}$. W.l.g assume that $x_1'=x_3, x'_2=x_4, x'_3=x_5$ and $ x'_4=x_6$. Now consider $X''=\{x_5, \ldots,x_9\}$. If there exists a vertex of $Y\setminus W$ say $y$, such that   $|N_G(y)\cap X''|\leq 1$, then it is easy to say that $K_{4,4}\subseteq \overline{G}[X'',\{y_1,y'_1,y'_2,y\}]$. So suppose that  $|N_G(y)\cap X''|\geq 2$ for each $Y\setminus W$. Therefore as $|Y\setminus W|=10$, $N_G(y'_3)=\{x_5,x_6\}$, and $|X''|=5$, one can say that $|E(G[X'',Y])|\geq 22$, that is there exist at least two vertices of $X''$ with degree at least $5$, which is impossible. Therefore, the proof of  Case 1 is complete.
		
		{\bf Case 2: $|N_G(y_1)\cap A'|= 1$}. W.l.g we may assume that $N_G(y_1)=\{x_1,x_2,x_3\}$ also let $N_G(x_3)=Y_3=\{y_1,y_{10},y_{11},y_{12}\}$. Since $|N_G(x)\cap Y_i'|\leq 1$ for each $x\in A'$ and each $i\in \{1,2\}$, we have $|N_G(x)\cap Y'_4|\geq 1$ for each $x\in A'\setminus \{x_3\}$, where $Y'_4=\{y_{13},y_{14}\}$. Set $Y'_i=Y_i\setminus\{y_1\}$ for $i=1,2$ and $Y'_3=\{y_{10},y_{11},y_{12}\}$. Hence we have the following claim.
		\begin{claim}\label{c17}
			We have:
			\begin{itemize}
				\item {\bf (P1):}  $|N_G(y)\cap X\setminus\{x_1,x_2,x_3\}|=3$ for each $y\in Y'_4=\{y_{13}, y_{14}\}$.  
				\item {\bf (P2):}  $|N_G(x)\cap Y'_i|=1$ for each $x\in A'\setminus\{x_3\}$ and each $i\in\{1,2,3\}$.	
			\end{itemize}
			
		\end{claim}	
		\begin{proof}[Proof of Claim \ref{c17}] 
			Prove $(P1)$: By contradiction, w.l.g suppose that  $|N_G(y_{13})\cap X\setminus\{x_1,x_2,x_3\}|\geq 4$. W.l.g assume that $A''=\{x_4,x_5,x_6,x_7\}\subseteq N_G(y_{13})$. If $|N_G(x)\cap \{y_{14}\}|= 0$, then one say that $|N_G(x)\cap Y'_3|= 1$ for each $x\in A''$, and since $|A''|=4$ and $|Y'_3|=3$, it is easy to say that $K_{2,2}\subseteq G[A'', \{y,y_{13}\}]$ for some $y\in Y'_3$. So suppose that $\{y_{14}\}\in N_G(x)$ for one $x\in A''$. W.l.g assume that $y_{14}\in N_G(x_4)$. If $|N_G(x_4)\cap Y'_3|= 1$ then the proof is same. So suppose that $|N_G(x_4)\cap Y'_3|= 0$, and w.l.g assume that $N_G(x_4)\cap Y'_1= \{y_2,y_6, y_{13}, y_{14}\}$. Now set $B= X\setminus\{x_1,x_2,x_3,x_4\}$ and $B'=\{y_1,y_2,y_3\}$.  As $|N_G(y_1)\cap A'|= 1$, and for each $x\in X\setminus\{x_1,x_2\}$, $|N_G(x)\cap Y'_4|\geq 1$, therefore we have $K_{5,3}\cong [B, B']\subseteq \overline{G}$. Hence if there exist a vertex of $Y\setminus B'$ say $y$ such that $|N_{\overline{G}}(y)\cap B|\geq 4$, then  $K_{4,4}\cong [B, B'\cup \{y\}]\subseteq \overline{G}$. So let $|N_{\overline{G}}(y)\cap B|\leq 3$, that is $|N_{G}(y)\cap B|\geq 2$ for each $y\in Y\setminus B'$. Now as $|Y\setminus B'|=11$ one can said that $ |E(G[B, Y\setminus B'])|\geq 22$. Therefore as $|B|=5$, one can say that there exists at least two vertices of $B$ say $x,x'$ such that $|N_G(x)|=|N_G(x')|=5$, a contradiction to $|A=\{x_1,x_2\}|=2$. 
			
			To prove $(P2)$, if for one $x\in A'\setminus\{x_3\}$ and one $i\in\{1,2,3\}$, $|N_G(x)\cap Y'_i|=0$, then  $|N_G(x)\cap Y'_4|= 2$. Hence as $|A'\setminus \{x_3\}|=6$, and $|N_G(x)\cap Y'_4|\geq 1$ for each $x\in A'\setminus \{x_3\}$, then it can be cheeked that  $|N_G(y)\cap X\setminus\{x_1,x_2,x_3\}|\geq 4$ for one  $y\in Y'_4$, and the proof is complete by part $(P1)$.
		\end{proof}	 
		
		Now by Claim \ref{c17}, w.l.g suppose that $R=N_G(y_{13})\cap X\setminus\{x_1,x_2,x_3\}= \{x_4,x_5,x_6\}$  and $R'=N_G(y_{14})\cap X\setminus\{x_1,x_2,x_3\}= \{x_7,x_8,x_9\}$.  Hence by $(P1)$ and $(P2)$ and  w.l.g we can suppose the following.
		\begin{itemize}

			\item {\bf (P3):}  $N_G(x_{4})= \{y_2,y_6,y_{10}, y_{13}\}$.
			\item {\bf (P4):}  $N_G(x_{5})= \{y_3,y_7,y_{11}, y_{13}\}$. 
			\item {\bf (P5):}  $N_G(x_{6})= \{y_4,y_8,y_{12}, y_{13}\}$. 	
		\end{itemize}  
		Now consider $R'$, by $(P2)$, w.l.g let $y_{10}\in N_G(x_7)$,  $y_{11}\in N_G(x_8)$, and  $y_{12}\in N_G(x_9)$. As $|R'|=3$ and $|Y'_1|=|Y'_2|=4$ and by $(P2)$, it can be said that the following  properties are established.
		\begin{itemize}
			\item {\bf (P6):} There exists a vertex of $R'$ say $x$, so that $|N_G(x)\cap Y'_1\setminus\{y_5\}|= |N_G(x)\cap Y'_2\setminus\{y_9\}|=1$.	  	
		\end{itemize}
		Therefore, by $(P6)$ w.l.g we may suppose that $x=x_7$, $N_G(x_{7})\cap Y'_1\setminus\{y_5\}=\{y'\}$ and  $N_G(x_7)\cap Y'_2\setminus\{y_9\}=\{y''\}$. Since $y_{10}\in N_G(x_7)$, it can be said that $y'\neq y_2$ and $y''\neq y_6$, otherwise either $K_{2,2}\subseteq G[ \{x_4,x_7\}, \{y_2,y_{10}\}]$ or $K_{2,2}\subseteq G[ \{x_4,x_7\}, \{y_6,y_{10}\}]$, a contradiction. Hence w.l.g assume that $y'=y_3$. As $y_3,y_7\in N_G(x_5)$, one can say that $y''=y_8$. Therefore we have the following.
		\begin{itemize}
			\item {\bf (P7):}   $N_G(x_{7})= \{y_3,y_8,y_{10}, y_{13}\}$. 
		\end{itemize}
		Now by considering $(P4), (P5), (P6)$, and $P(7)$ and since $N_G(x_3)\cap (Y'_1\cup Y'_2)=\emptyset$, one can say that  $K_{4,4}\subseteq \overline{G}[\{x_3,x_5,x_6,x_7\}, \{y_2,y_5,y_6,y_9\}]$. Which means that the proof of Case 2 is complete.
		
		Therefore By Case 1 and Case 2, the proof is complete.
	\end{proof}	
	In the following  theorem,	 we find the values of $BR_m(K_{2,2}, K_{3,3})$ for each $m\in \{10,\ldots,13\}$.
	
	\begin{theorem}\label{th7}
		$BR_m(K_{2,2}, K_{4,4})=14$ for each $m\in \{10,11,12,13\}$.
	\end{theorem}
	\begin{proof}
	By Theorem \ref{th1} as  $BR(K_{2,2}, K_{4,4})=14$ 	it is sufficient to show that  $BR_{10}(K_{2,2}, K_{4,4})= 14$. Suppose that $G$ be a subgraph of $K_{10,14}$, hence as  $|E(K_{10,14})|=140$, then either $|E(G)|\geq  43$ or $|E(\overline{G})|\geq 98$. Therefore by Lemma \ref{l0} as 	$z((10,14), K_{2,2})\leq  42$ and $z((10,14), K_{4,4})\leq  97$, it can be said that either $K_{2,2}\subseteq G$ or $K_{4,4}\subseteq \overline{G}$. Therefore, $BR_m(K_{2,2}, K_{4,4}) = 14$, for each $m\in \{10,11,12,13\}$. 
	\end{proof}
	
	\begin{proof}[\bf Proof of Theorem \ref{M.th}]
		For $m=2,3,4$, it is easy to say that $BR_m(K_{2,2}, K_{3,3})$ does not exist. Now,
		by combining Theorems \ref{th1}, \ref{th2}, \ref{th3}, \ref{th4}, \ref{th5}, \ref{th6},   and \ref{th7},  we conclude that the proof of Theorem \ref{M.th} is complete.
	\end{proof}
 
	%%%%%%%%%%%%%%%%%%%%%%%%%%%%%%%%%%%%%%%%%
	\bibliographystyle{spmpsci} 
	\bibliography{BI}
\end{document}